\documentclass[a4paper,11pt]{amsart}
\usepackage{amssymb,amsmath,pdfsync,amsthm,array,amsfonts}

\newtheorem{thm}{Theorem}[section]

\newtheorem{remark}[thm]{Remark}
\newtheorem{lem}[thm]{Lemma}
\newtheorem{lemma}[thm]{Lemma}

\newtheorem{corollary}[thm]{Corollary}

\newtheorem{dfn}[thm]{Definition}

\newcommand{\SL}{\mathrm{SL}}
\newcommand{\PSL}{\mathrm{PSL}}
\newcommand{\SU}{\mathrm{SU}}
\newcommand{\PSU}{\mathrm{PSU}}

\newcommand{\Sp}{\mathrm{Sp}}
\newcommand{\PSp}{\mathrm{PSp}}

\newcommand{\SO}{\mathrm{SO}}
\newcommand{\PSO}{\mathrm{PSO}}
\newcommand{\POm}{\mathrm{P\Omega}}
\newcommand{\Om}{\mathrm{\Omega}}
\newcommand{\GO}{\mathrm{GO}}

\newcommand{\Gbar}{\overline{G}}
\newcommand{\Tbar}{\overline{T}}
\newcommand{\GF}{G}
\newcommand{\kodd}{k_{odd}}
\newcommand{\keven}{k_{even}}
\newcommand{\godd}{g_{odd,r}(n)}
\newcommand{\geven}{g_{even,r}(n)}
\newcommand{\frfn}{f_{r}}
\newcommand{\fr}{f_{r}(n)}
\newcommand{\goddfn}{g_{odd,r}}
\newcommand{\gevenfn}{g_{even,r}}
\newcommand{\B}{C_0}
\newcommand{\hxr}{h_{X,r}}

\renewcommand{\l}{n}
\makeatletter
\def\imod#1{\allowbreak\mkern7mu({\operator@font mod}\,\,#1)}
\makeatother

\title[Proportions of $r$-regular elements in finite classical groups]{Proportions of $r$-regular elements in finite classical groups}
\author{L\'{a}szl\'{o} Babai}
\address{Department of Computer Science,
University of Chicago,
1100 East 58th Street,
Chicago, IL
60637, USA}
\email{laci@cs.uchicago.edu}

\author{Simon Guest}
\address{Centre for Mathematics of Symmetry and Computation (M019), 
The University of Western Australia, 35 Stirling Hwy,
Crawley, WA 6009, Australia\footnote{Current address for S. Guest: School of Mathematics, University of Southampton, Southampton, SO17 1BJ, UK}}
\email{simon\_guest@baylor.edu}

\author{Cheryl E. Praeger}
\address{Centre for Mathematics of Symmetry and Computation (M019), 
The University of Western Australia, 35 Stirling Hwy,
Crawley, WA 6009, Australia}
\address{\textit{Also affiliated with:}}
\address{King Abdulaziz University, Jeddah, Saudi Arabia
}
\email{cheryl.praeger@uwa.edu.au}

\author{Robert A. Wilson}
\address{School of Mathematical Sciences,
Queen Mary, University of London,
Mile End Road,
London, E1 4NS, United Kingdom}
\email{r.a.wilson@qmul.ac.uk}

\begin{document}

\maketitle

\begin{abstract}
For a prime $r$, we obtain lower bounds on the proportion of $r$-regular elements in classical groups 
and show that these lower bounds are the best possible lower bounds that do not depend on the order 
of the defining field. Along the way, we also provide new upper bounds and answer some open 
questions of the first author, P\'{a}lfy and Saxl.
\end{abstract}
\section{Introduction}

A number of results have appeared recently giving lower bounds for the proportion of $r$-regular 
elements, for a prime $r$, in a group of Lie type $G$ in characteristic $p \ne r$. We denote this 
proportion by $p_r(G)$. Of particular interest is the dependence on the dimension, in the case of 
classical groups in dimension $d$. In \cite{BPS}, the first author, P\'{a}lfy and Saxl show that 
the proportion of $r$-regular elements in any classical group is at least $1/2d$. To complement 
their lower bounds, they also showed that for all prime powers $q \equiv -1 \imod{4}$ and all $d 
\ge 2$ such that $(d,q-1) \le 2$, the proportion of $2$-regular elements in $\PSL_d(q)$ is
$p_2(\PSL_d(q)) \le 4q^{-1} + 4 (\pi d)^{-1/2}$ (see \cite[Theorem 6.1]{BPS}).  In a remark 
after Corollary 22 in  \cite{PW}, Parker and the fourth author show that, in general, a lower 
bound of $1/d$ for $p_2(\PSL_d(q))$ would be best possible. On the other hand, in \cite{GP}, 
the second and third authors improved these $O(d^{-1})$ lower bounds to $O(d^{-3/4})$ in the 
case $r=2$ for symplectic and orthogonal groups. We generalize and  improve these results here, 
obtaining lower bounds and upper bounds in terms of the dimension. 

\begin{thm} \label{t:simplelb}
{\rm (a)} If $G = \PSL_d(q)$ or $\PSU_d(q)$ then 
\[ p_r(G) \ge \frac{1}{d}.\]
{\rm (b)} Let $G = \Sp_{d}(q)$, $\SO^{\pm}_d(q)$, $\SO^{\circ}_d(q)$, $\Om^{\pm}_d(q)$, 
$\Om^{\circ}_d(q)$ or one of the corresponding projective groups.
 There exist explicit constants $C_1$, $C_2$ such that
 \[ p_r(G) \ge \begin{cases}
  C_1 d^{-3/4} & \mbox{ if $r =2$;}\\
  C_2 d^{-1/2} & \mbox{ if $r$ is odd.}
 \end{cases}\]
\end{thm}
We show that these lower bounds are essentially best possible for infinitely many orders of the 
underlying field (and indeed infinitely many fields in infinitely many characteristics).  
\begin{thm} \label{t:simplebestlb}
{\rm (a)} Let $G = \PSL_d(q)$ or $\PSU_d(q)$ and let $\epsilon >0$. Then 
there exist infinitely many primes, and infinitely many powers $q$ of each, for which $p_r(G) \le \frac{1}{d} + \epsilon.$ 

{\rm (b)} Let $G = \Sp_{d}(q)$, $\SO^{\pm}_d(q)$, $\SO^{\circ}_d(q)$, $\Om^{\pm}_d(q)$, $\Om^{\circ}_d(q)$ or one of the corresponding projective groups.
 There exist explicit constants $C_3$, $C_4$, infinitely many primes, and infinitely many powers $q$ of  each, for which
  \[ p_r(G) \le \begin{cases}
  C_3 d^{-3/4} & \mbox{ if $r =2$;}\\
  C_4 d^{-1/2} & \mbox{ if $r$ is odd.}
 \end{cases}\]
\end{thm}
In addition, we prove some general upper bounds in the case $r=2$.
At this stage it is not clear to us whether these are essentially best possible.
\begin{thm} \label{t:simpleub1}
{\rm (a)} Let $G = \PSL_d(q)$ or $\PSU_d(q)$. Then, 
for each odd $q$, there exists an explicit constant $C_q$ such that for all $d \ge 2$ we have 
\[ p_2(G) \le  \frac{C_q}{\sqrt{d}}.\] 

{\rm (b)} Let $G = \Sp_{d}(q)$, $\SO^{\pm}_d(q)$, $\SO^{\circ}_d(q)$, $\Om^{\pm}_d(q)$, $\Om^{\circ}_d(q)$ or one of the corresponding projective groups.
Then there exist constants $C_5$, $C_6$ such that for all odd $q$ and all $d \ge 4$
\[ p_2(G) \le \begin{cases}
  C_5 d^{-5/8} & \mbox{ if $q \equiv 1 \imod{4}$;}\\
  C_6 d^{-1/2} & \mbox{ if $q \equiv -1 \imod{4}$.}
  \end{cases}\]
\end{thm}
We note that we can take $C_q = 2(q-1)_2/\sqrt{\pi }$ when $G=\PSL_d(q)$ and $C_q= 2(q+1)_2/\sqrt{\pi }$ when $G=\PSU_d(q)$, where $(q \pm1)_2$ denotes the largest power of $2$ dividing $(q \pm 1)$. In fact, we sharpen our bounds considerably in Theorem \ref{t:main} below.
As a notational convenience, we will frequently write the dimension $d$ of the natural module for $G$ in terms of the (untwisted) Lie rank $n$. In order to express our sharper lower bounds more succinctly, we define  the following functions of $n$:
 \begin{align} \label{e:frfact}
 \fr = \begin{cases}
  \frac{(2n)!}{2^{2n}n!^2}& \hbox{if $r$ is odd;}  \\
  \frac{\Gamma(n+ 1/4)}{\Gamma(1/4) \Gamma(n+1)}  & \hbox{if  $r=2$;}  
      \end{cases}  
\end{align} 
 \begin{align} \label{e:goddgamma}
 \godd  =  \begin{cases}
 \frac{(2n)!}{2^{2n}n!^2} \cdot \left( \frac{2n}{2n-1}\right) & \hbox{if $r$ is odd;}  \\
  \frac{\Gamma(n+ 1/4)}{\Gamma(1/4) \Gamma(n+1)} + \frac{\Gamma(n- 1/4)}{4\Gamma(3/4) \Gamma(n+1)}  & \hbox{if  $r=2$;}  
      \end{cases}  
\end{align} 
and
 \begin{align}  \label{e:gevengamma}
 \geven =  \begin{cases}
 \frac{(2n)!}{2^{2n}n!^2}\cdot  \left( \frac{2n-2}{2n-1}\right)& \hbox{if $r$ is odd;}  \\
  \frac{\Gamma(n+ 1/4)}{\Gamma(1/4) \Gamma(n+1)} - \frac{\Gamma(n- 1/4)}{4\Gamma(3/4) \Gamma(n+1)}  & \hbox{if  $r=2$.}  
      \end{cases}  
\end{align}
See section 2 for details.  In particular, equations \eqref{e:godd} and \eqref{e:geven} provide some insight into our choice of notation for the functions $\goddfn$ and $\gevenfn$.

\begin{thm} \label{t:main}
Let $q$ be a  prime power, $r$ a prime not dividing $q$ and $\epsilon \in \mathbb{R}$ with $0< \epsilon <1$. Let $G = X_d(q)$ be a finite classical group with natural module of dimension $d$ and (untwisted) Lie rank $n$. 
 Let $p_r(G)$ denote  the proportion of $r$-regular elements in $G$.
 Suppose that $X$, $d$  and $\hxr$ are defined in one of the rows of  Table \ref{tab:main}.\\
 \emph{(i)} Then 
 \[p_r(G) \ge \hxr(n).\]
  \emph{(ii)} For infinitely many primes $p$, there exist infinitely many  powers $q$ of $p$ for which
   \[p_r(G) < \hxr(n) + \epsilon.\] \emph{(iii)} Moreover,  $p_r(G/Z(G))  = |Z(G)|_rp_r(G)$.\\
 \begin{table}[htdp] 
\begin{center}\begin{tabular}{ccc}\hline $X_d$ & {\rm Conditions}  & $\hxr(n)$ \vspace*{0.05cm} \\
\hline \\ [-2.5ex]
  $\PSL_{n+1}$ &  & $1/(n+1)$ \\
  $\PSU_{n+1}$ &  & $1/(n+1)$ \\
$\Sp_{2n}$ &  & $\fr$ \\
 $\SO^{\circ}_{2n+1}$ &  & $\fr$ \\
   $\Om^{\circ}_{2n+1}$ & & $(2,r)\fr$ \\
    $\SO^+_{2n}$ & & $\geven$\\
    $\Om^+_{2n}$ &  & $(2,r)\geven$ \\
     $\SO^-_{2n}$ &  {\rm $n$ even} & $\godd$ \\
        $\SO^-_{2n}$ &  {\rm $n$ odd} & $\geven$ \\
        $\Om^-_{2n}$ &  {\rm $n$ even} & $(2,r)\godd$ \\
           $\Om_{2n}^-$ &  {\rm $n$ odd} & $(2,r)\geven$   \vspace*{0.05cm}\\
          \hline
           \end{tabular} 
             \vspace*{0.15cm}
           \caption{Lower bounds $\hxr(n)$ on $p_r(X_d(q))$ }
           \label{tab:main}
\end{center}
\end{table}
           

\end{thm}
\begin{remark} \label{rem:main}
 \emph{(a) For fixed $X$, $d$ and $\epsilon$, we show, in many cases,  that for all primes $p$ (distinct from $r$), there exist infinitely many powers $q$ of $p$ for which $p_r(X_d(q)) < \hxr(n)+\epsilon$. See Lemmas \ref{l:LUB}, \ref{l:UUB}, \ref{l:SpUB} and \ref{l:SOUB} for details.}\\
 \emph{(b) For each row of Table  \ref{tab:main}, the function $\hxr$ is the same for every \emph{odd} prime $r$.} \\
    \emph{(c) Much better lower bounds are possible for certain values of $q$ and $r$. For example, if $G=\PSL_d(q)$ and the multiplicative order $m$ of $q$ modulo $r$ is at least $2$, then $p_r(G) \ge Cd^{-1/m}$, where $C$ is an explicit constant. See Lemmas \ref{l:psl1}, \ref{l:psu1}, \ref{l:Sp} and \ref{l:SO} for details.} \\
    \emph{(d) The lower bound $p_r(G) \ge 1/2d$ in \cite{BPS} also holds when $r|q$, whereas our results do not deal with this case. However, Guralnick and Lubeck \cite{GuLu} have shown that, when $r|q$, we have $$p_r(X_n(q)) \ge  1- 3/(q-1)-2/(q-1)^2$$ for all simple groups of Lie type $X_n(q)$. It follows that the bounds in Table \ref{tab:main} hold when $r|q$ unless $q$ is small. In any case, the proof of the lower bound $p_r(G) \ge 1/2d$ in \cite{BPS} actually shows that $p_r(G) \ge 1/d$ when $r |q$; thus the lower bounds in Table \ref{tab:main} for $\PSL_d(q)$ and $\PSU_d(q)$ hold for all values of $r$ and $q$.}
\end{remark}

Remark \ref{rem:main}(c) above answers an open question in \cite[Section 8]{BPS}; moreover the following theorem answers the main open question in \cite[Section 8]{BPS}. See Section \ref{openqs} for details, the answer to a further open question and for the proof of Theorem \ref{t:limsup}.
\begin{thm} \label{t:limsup}
For all $ \epsilon >0$, there exists a prime power $q$ and a constant $C$ such that, for all $d\geq2$, 
\[p_2(\PSL_d(q)) \le C d^{\epsilon-1}.\]
\end{thm}

We prove Theorem \ref{t:main} for linear, unitary, symplectic and odd-dimensional orthogonal, and even-dimensional orthogonal groups in sections $3$, $4$, $5$ and $6$ respectively.  Theorem \ref{t:simplelb} is a direct consequence of Theorem \ref{t:main}. Theorems \ref{t:simplebestlb} and \ref{t:simpleub1} are proved in Sections $3$, $4$, $5$, and $6$.  In particular, the constants $C_1$, $C_2$, $C_3$, $C_4$ involved in the statements can be obtained from Corollaries \ref{c:S} and \ref{c:caseO}. \\

\section{Preliminaries}
Let $\Gbar$ be a connected reductive algebraic group defined over $\bar{\mathbb{F}}_q$, the algebraic closure of a field $\mathbb{F}_q$ of order $q$.  Let $F$ be a Frobenius morphism of $\Gbar$, and let $G=\Gbar^F$ be the subgroup of $\Gbar$ fixed elementwise by $F$, so that $G$ is a finite group of Lie type. Let $\Tbar$ be an $F$-stable maximal torus in $\Gbar$ and let $W:=N_{\Gbar}(\Tbar)/\Tbar$ denote the Weyl group of $\Gbar$. We will say that two elements $w, w^{\prime}$ in $W$ are $F$-conjugate if there exists $x$ in $W$  such that $w^{\prime}= x^{-1}wF(x)$. This is an equivalence relation, and we will refer to the equivalence classes as \textit{$F$-classes}. Moreover, there is an explicit one-to-one correspondence between the $F$-classes of $W$ and the $\GF$-conjugacy classes of maximal tori in $\GF$. A thorough description of the correspondence can be found in \cite{NiePra}, and we will summarize the necessary results below. Although the correspondence is between $F$-classes in $W$ and $\GF$-conjugacy classes of maximal tori in $\GF$, we will frequently refer to an \textit{element} of $W$ corresponding to a maximal torus in $\GF$, rather than a $\GF$-conjugacy class of tori, where there is little possibility of confusion.

Now recall that every element $g$ in $\GF$ can be expressed uniquely in the form $g=su$, where $s \in \GF$ is semisimple, $u \in \GF$ is unipotent and $su=us$. This is the multiplicative Jordan decomposition of $g$ (see \cite[p. 11]{Carter2}). We now define a quokka set as a subset of a finite group of Lie type that satisfies certain closure properties.

\begin{dfn}
{\rm Suppose that $\GF$ is a finite group of Lie type. A nonempty subset $Q$ of $\GF$ is called a \textit{quokka set}, or quokka subset of $\GF$,  if the following two conditions hold.
\begin{enumerate}
 \item[(i)] For each $g \in \GF$ with Jordan decomposition $g=su=us$, where $s$ is the semisimple part of $g$ and $u$ the unipotent part of $g$, the element $g$ is contained in $Q$ if and only if $s$ is contained in $Q$; and
\item [(ii)] the set $Q$ is a union of $\GF$-conjugacy classes.
\end{enumerate}}
\end{dfn}
For $\GF$ a classical group, and a prime $r$ not dividing $q$, define the subset
\begin{displaymath}
Q(r,\GF) := \{ g \in \GF \,: \, r\,\nmid |g| \},
\end{displaymath}
consisting of all the $r$-regular elements $g$ in $\GF$.  We readily see that $Q(r,\GF)$ is a quokka set.



Applying \cite[Theorem 1.3]{NiePra} we have

\begin{lem} \label{thm:maintool}
Let $\GF$, $W$, $Q(r,\GF)$ be as above, let $C$ be a subset of $W$ consisting of a union of $F$-classes of $W$. For each $F$-class $\B$ in $W$, let $T_{\B}$ denote a maximal torus corresponding to $\B$.
Then
\begin{align} \label{eqn:maintool1}
\frac{|Q(r,\GF)|}{|\GF|} &= \sum_{\B \subset W}\frac{|\B|}{|W|}.\frac{|T_{\B} \cap Q(r,\GF)|}{|T_{\B}|},
\end{align}
and if there exists a constant $A \in [0,1]$ such that $\frac{|T_{\B} \cap Q(r,\GF)|}{|T_{\B}|} \ge A$ for all $F$-classes in $C$, then
\begin{align} \label{eqn:maintool2}
\frac{|Q(r,\GF)|}{|\GF|} \ge\frac{A|C|}{|W|}.
\end{align}
\end{lem}
In order to obtain an estimate for the proportions of $r$-regular elements in $\PSL_n(q)$ and $\PSU_n(q)$ we prove the following lemma, which essentially follows the proof of  \cite[Theorem 1.6]{NiePra}.

\begin{lem} \label{QZ}
Let $G$ be a classical group, $Q = Q(r,G)$, $Z = Z(G)$ and $QZ:=\{xz \,:\, x \in Q, z \in Z\}$. Then the proportion of $r$-regular elements in $G/Z$ is equal to 
\[ p_r(G/Z) = \frac{|QZ|}{|G|} =  \frac{|Q||Z|}{|G||Q \cap Z|}.\]
\end{lem}
\begin{proof} 
  First observe that the set of $r$-regular elements in $G/Z$ is equal to $QZ/Z := \{xZ \,:\, x \in Q\}$. Clearly if $x \in Q$ then $xZ$ is an $r$-regular element in $G/Z$. Conversely, suppose $xZ$ is an $r$-regular element in $G/Z$. Let $m$ be the order of $xZ$ and $r^k=|x|_r$ the $r$-part of the order of $x$. Since $r^k$ and $m$ are coprime, there exist $a,b \in \mathbb{Z}$ such that $1=ar^k+bm$. But then $x =x^{ar^k} x^{bm}$ and $x^{ar^k} \in Q$ and $x^{bm} \in Z$.  So $x Z  = x^{ar^k}Z \in QZ/Z$.
  
 So the proportion of $r$-regular elements in $G/Z$ is equal to 
 \begin{displaymath} 
  \frac{|QZ/Z|}{ |G/Z|} = \frac{ |QZ|}{|G|}.
 \end{displaymath}
 To calculate $|QZ|$, we count the set $A:= \{( z,g,h) \,:\, z \in Z, g \in Q, h \in QZ \mbox{ and } gz=h\}$ in two ways. On the one hand, $|A| = |Z| |Q|$ since each pair $(z,g) \in Z \times Q$ occurs exactly once in $A$ and there are $|Z||Q|$ such pairs. On the other hand, for each $h \in QZ$, there exist at least one $z \in Z$ and $g \in Q$ such that $h=gz$. If $z' \in Z$ and $g' \in Q$ also satisfy $h=g'z'$ then $g^{-1}g' = (z^{\prime})^{-1}z \in Z$. And if we let $y := g^{-1}g' = (z^{\prime})^{-1}z $, then $z'=zy^{-1}$ and $g'  = g y$. Note that for all $y \in  Z$, the element $z'=zy^{-1}$  is in $Z$, but $g'  = g y$ is contained in $Q$ if and only if $y \in Q$. Therefore given $h \in QZ$, the number of pairs $(z,g) \in Z \times Q$ for which $h = gz$ is $|Q \cap Z|$, and thus $|A| = |QZ||Q \cap Z|$. Equating the two expression for $|A|$ gives 
 \begin{displaymath} 
  |QZ| = \frac{ |Q||Z|}{|Q \cap Z| }
 \end{displaymath}
 and the result follows.
  \end{proof}

We now state some number theoretic results, which we will find useful.
\begin{lem}\label{l:nt}
Let $q \ge 2$ be an integer. For all integers $i,j \ge 1$ we have 
\begin{align*} 
(q^i-1,q^j-1) & = q^{(i,j)}-1; \\
(q^i-1,q^j+1) & = \begin{cases}
 q^{(i,j)}+1 & \mbox{if } 2j_2\le i_2;\\
 (2,q-1)   & \mbox{otherwise; } 
\end{cases} \\
(q^i+1,q^j+1) & = \begin{cases}
 q^{(i,j)}+1 & \mbox{if } j_2= i_2;\\
 (2,q-1)   & \mbox{otherwise. }  
\end{cases}
\end{align*}  
\end{lem}

\begin{lem} \label{lem:2parts}
Let $q$ and $i$ be positive integers, and let $r$ be a prime not dividing $q$. If $q$ is odd then
 \begin{equation} \label{qi+1}
  (q^i+1)_2 =
  \left\{
               \begin{array}{ll}
                 2 & \hbox{if $i$ is even;} \\
                 (q+1)_2 & \hbox{if $i$ is odd;}
               \end{array}
             \right.
   \end{equation}
  and
 \begin{equation} \label{qi-1}
  (q^i-1)_2 =
  \left\{
               \begin{array}{ll}
                 (q-1)_2    & \hbox{if $i$ is odd;} \\
                 i_2(q-1)_2 & \hbox{if $q \equiv 1 \imod 4$, and $i$ is even;} \\
                 i_2(q+1)_2 & \hbox{if $q \equiv 3 \imod 4$, and $i$ is even.}
               \end{array}
             \right.
\end{equation}
If $r$ is odd then 
\begin{equation} \label{qi-1_r}
  (q^i-1)_r =
  \left\{
               \begin{array}{ll}
                 i_r(q-1)_r    & \hbox{if $r| q-1$;} \\
                 i_r(q+1)_r & \hbox{if $r | q+1$ and $i$ is even;} \\ 
                 1& \hbox{if $r | q+1$ and $i$ is odd.}
               \end{array}
             \right.
\end{equation}
and \begin{equation} \label{qi+1_r}
  (q^i+1)_r =
  \left\{
               \begin{array}{ll}
                 i_r(q+1)_r    & \hbox{if $r| q+1$ and $i$ is odd;} \\
                 1 & \hbox{if $r|q+1$ and $i$ is even;}   \\
                   1 & \hbox{if $r | q-1$.}                \end{array}
             \right.
\end{equation}

\end{lem}
\begin{proof}
Equations \eqref{qi+1} and \eqref{qi-1} are proved in \cite[Lemma 2.5]{GP}.
To prove equation \eqref{qi-1_r}, Lemma 2.8 of \cite{pabundant} implies that if $r|q-1$, then for all integers $i$, we have
\[(q^{r^j} -1)_r = r^j (q-1)_r.\]
In particular, for all positive integers $i$ we have $(q^{i_r} -1)_r = i_r(q-1)_r$.
And writing $i=i_rt$, where $(r,t)=1$, we have
\[(q^i-1) =  (q^{i_r} -1)(q^{i_r(t-1)}+ \cdots + q^{i_r} +1). \]
But $q^{i_r(t-1)}+ \cdots + q^{i_r} +1 \equiv t \imod{r}$ and so $(q^i-1)_r=(q^{i_r}-1)_r= i_r(q-1)_r$, which proves the first line of \eqref{qi-1_r}. To prove the second line, suppose $r|q+1$ and $i$ is even. Then $r | q^2-1$ and 
\[(q^i-1)_r  = ((q^2)^{i/2}-1)_r = (i/2)_r(q^2-1)_r \]
by the first line of \eqref{qi-1_r}. Moreover, $(i/2)_r(q^2-1)_r = i_r (q+1)_r$ since $r$ is odd and $r | q+1$. This proves the second line. Now suppose $r | q+1$ and $i$ is odd. In particular, $q \equiv -1 \imod{r}$, and since $i$ is odd, we have $q^i \equiv -1 \imod{r}$;  hence $(q^i-1)_r=1$. This proves the third line of \eqref{qi-1_r}.
Now equation \eqref{qi+1_r} follows from  \eqref{qi-1_r} since the $r$-part of $q^i+1$ is $(q^{2i}-1)_r/(q^i-1)_r$.
\end{proof}

We will also need bounds on the proportion of elements in the symmetric group $S_d$ that have cycles of certain lengths. We define $s_{\neg m}(d)$ to be the proportion of elements in $S_d$ that have no cycles of length divisible by $m$. 
Erd\H{o}s and Tur\'an 
\cite{ET2} proved that when $m$ is a prime power, we have the formula 
\begin{equation} \label{product}
  s_{\neg m}(d) = \prod_{i=1}^{\lfloor{d/m}\rfloor} \left(1- \frac{1}{im}\right).
\end{equation}
In \cite{BLNPS}, it is shown that  this formula also holds when $m$ is not a prime power, and moreover that there exist constants $c_m$ (depending on $m$) such that
\begin{equation} \label{beals}
c_m d^{-1/m}\left (1- \frac{1}{d}\right ) \le s_{\neg m}(d) \le c_m d^{-1/m}\left (1+ \frac{2}{d}\right ). 
 \end{equation}
 Our functions $\goddfn$ and $\gevenfn$ involve the Gamma function $\Gamma$, which is defined for all $z \in \mathbb{C}$, with ${\mathrm Re}\,z >0$, by the equation
\begin{align*} 
\Gamma(z):=\int_0^{\infty} y^{z-1}e^{-y} \, dy.
\end{align*}
Recall that for all positive integers $n$ we have
\begin{align} \label{e:gammafac}
  \Gamma(n+1) & = n!
\end{align}
We denote  the number of permutations in $S_n$ with precisely $k$ cycles by $c(n,k)$. These numbers are known as the unsigned Stirling numbers of the first kind (see section 4 of \cite{GP} for details).
 Using generating function methods (see (4.4) and (4.6) of \cite{GP}), we can prove that 
\begin{equation} \label{e:fr}
\begin{split}
 \sum_{k=1}^n \frac{c(n,k)}{n! 2^k}  &= \frac{(2n)!}{2^{2n}n!^2} = \fr  \hbox{ for $r=2$;} \\
 \sum_{k=1}^n \frac{c(n,k)}{n! 4^k}  &= \frac{\Gamma(n+1/4)}{\Gamma(1/4) \Gamma(n+1)} = \fr \hbox{ for $r$ odd}. 
\end{split}
 \end{equation}

More generally, \cite[(4.3)]{GP} and \cite[6.1.22]{AS} show that for $-1 < x <1$ we have 
\begin{equation} \label{e:genstirlingsum}
  \sum_{k=1}^{n} \frac{c(n,k) x^k}{n!}  = \frac{1}{n!}\prod_{k=0}^{n-1}(x+k)  = \frac{\Gamma(n+x)}{\Gamma(x) \Gamma(n+1)}.
\end{equation}
To analyze the asymptotics of \eqref{e:genstirlingsum}, we shall make use of an inequality due to Ke{\v{c}}ki{\'c} and Vasi{\'c} \cite{KV}. It states that if $y > x \ge 1$ then 
\begin{equation} \label{e:KV}
  \frac{x^{x-1/2}}{y^{y-1/2}} e^{y-x} < \frac{\Gamma(x)}{\Gamma(y)} < \frac{x^{x-1}}{y^{y-1}} e^{y-x}.
\end{equation}   
Applying the inequality \eqref{e:KV} to \eqref{e:genstirlingsum} implies that for a given $x$ in $(-1,1)$ we have
\begin{equation} \label{e:sumcnkbound}
 \sum_{k=1}^{n} \frac{c(n,k) x^k}{n!} \le \frac{(n+x)^{n+x-1} e^{1-x}}{(n+1)^{n} \Gamma(x)} \le \left (1+ \frac{x}{n} \right)^{n} (n+x)^{x-1}e^{1-x} \le C_x n^{x-1}
  \end{equation} 
  for some constant $C_x$.
We also note that equations (4.8), (4.9), (4.10) and (4.11) of \cite{GP} give us 
\begin{equation} \label{e:godd}
  2\sum_{\substack{{k=1}\\{k \text{ odd}}}}^{n} \frac{c(n,k)}{n!2^k(2,r)^k} =   \godd 
\end{equation}
and 
\begin{equation} \label{e:geven}
   2\sum_{\substack{{k=1}\\{k \text{ even}}}}^{n} \frac{c(n,k)}{n!2^k(2,r)^k}  = \geven.
\end{equation}

Moreover, it will  be useful to know which of these functions provides the best lower bound, which is the purpose of the following two lemmas:
\begin{lem} \label{l:npi}
For all integers $ n\ge 2$ we have
   \begin{align} \label{eqn:2ellfactorial}
\frac{1}{\sqrt{\pi n}} \ge \frac{(2 \l)!}{ 2^{2\l} \l !^2} &\ge \frac{25}{29 \sqrt{\pi\l}}.
\end{align}
\end{lem}
\begin{proof} 
The inequalities are easily verified using  the  Stirling approximation
\[ \left | \frac{n!}{\sqrt{2\pi n} n^n e^{-n}} -1- \frac{1}{12n} \right | \le \frac{1}{288n^2} + \frac{1}{9940n^3},\]
which holds for all integers $n \ge 2$ (see \cite{Michel} for example).
%

 \end{proof}
 \begin{lemma} \label{l:goddineq}
  For all integers $n,m$ satisfying $ n \ge m \ge 2$, we have 
  \begin{align} 
    s_{\neg m}(n) \ge \godd > \fr > \geven.
  \end{align} 
 \end{lemma}
 \begin{proof} 
  The inequalities  $\godd > \fr$  and $\fr > \geven$ are clear from the definitions of the functions.  To prove that $ s_{\neg m}(n) \ge \godd$, we note that $ s_{\neg m}(n) \ge c_m n^{-1/m}(1-1/n)$ by \eqref{beals}. Moreover, Remark 2.2 of \cite{GPsn} implies that for all $m \ge 2$, we have  $c_m \ge c_2 = (\pi/2)^{-1/2}$. Therefore for all $n,m$  satisfying $n \ge m \ge 2$, we have  
  \[  s_{\neg m}(n) \ge c_2 n^{-1/2}\left (1- \frac{1}{n} \right ) \ge \left (\frac{2}{n\pi}\right )^{1/2} \left (1- \frac{1}{n} \right ). \]
  But Lemma \ref{l:npi} implies that $ \frac{2n}{(2n-1) (n \pi)^{1/2}} \ge \godd$ and it is easy to see that $2(1-n^{-1}) \ge \frac{2n}{(2n-1)}$ for $n \ge 5$. Thus $s_{\neg m}(n) \ge \godd$ for $ n \ge 5$ and it remains to check the cases where $n$ and $m$ satisfy $4 \ge n \ge m \ge 2$. We can do this  easily using \eqref{product}.
   \end{proof}

\section{Linear Case}

Suppose that $G =\SL_d(q)$. Then $W =S_d$ and an $F$-class in $W$ consisting of cycles of lengths $b_1, \ldots , b_m$ corresponds to a $\GF$-class of maximal tori $T$, which are subgroups  of 
\begin{equation*}
 \prod_{i=1}^m (q^{b_i}-1)
\end{equation*}
of index $q-1$. Set $Q=Q(r,G)$.

\begin{lemma} \label{l:psl1}
  Let $m$ be the multiplicative order of $q$ modulo $r$. Then $p_r(\SL_2(q)) \ge \frac{1}{2(q+1)_r}+ \frac{1}{2(q-1)_r}$ and for $d \ge 3$ we have
 \[
 p_r(\SL_d(q)) \ge \begin{cases}
 s_{\neg m}(d)  & \mbox{if }  m \ge 2;\\
   \frac{1}{d d_r} + \frac{1}{(d-1)(d-1)_r (q-1)_r} &  \mbox{if $m =1$ and $r$ is odd;}\\
    \frac{1}{d} + \frac{2}{(d-1)(d-1)_2 (q^2-1)_2}  & \mbox{if $r=2$ and $d$ is odd;}  \\
     \frac{2}{d d_2(q+1)_2} + \frac{1}{(d-1) (q-1)_2}  & \mbox{if $r=2$ and $d$ is even.}    
\end{cases}\]
Moreover, for $G = \PSL_d(q)$, the lower bound $p_r(G) \ge d^{-1}$ in part {\rm (i)} of Theorem \ref{t:main} holds.
  \end{lemma}
\begin{proof} 
First suppose that $r$ is odd and $d \ge 3$. Note that if $r | (q^{b_i}-1)$
 then $r | (q^{b_i}-1, q^m-1) = q^{(b_i,m)}-1$ and therefore $m | b_i$. It follows that if $w \in S_d$ has no cycle lengths divisible by $m$, then $r \nmid |T_w|$ and $|T_w \cap Q|/ |T_w| =1$.
 If $m \ge 2$ then using Theorem \ref{thm:maintool} we find that
 \begin{displaymath} 
  p_r(\SL_d(q)) = \frac{|Q|}{|\SL_d(q)|} \ge s_{\neg m}(d) \ge c_m d^{-1/m},
 \end{displaymath}
 where the second inequality follows from \eqref{beals}. Furthermore, the bounds in \eqref{beals} are sufficient to prove that $ s_{\neg m}(d) \ge 1/d$ unless $(d,m)=(3,2)$. But an easy direct calculation verifies that $s_{\neg 2}(3) \ge1/3$. The proportion of $r$-regular elements in $\PSL_d(q)$ when $m \ge 2$ is therefore 
 \[p_r(\PSL_d(q)) =  (d,q-1)_r p_r(\SL_d(q)) \ge d^{-1}\] 
 by Lemma \ref{QZ}.
 
 If $m =1$ then $r |q-1$. Consider the conjugacy classes  of $d$-cycles and $(d-1)$-cycles in $S_d$, which correspond to the $\GF$-classes of cyclic maximal tori in $\SL_d(q)$ of orders $(q^d-1)/(q-1)$ and $q^{d-1}-1$. Since $r |q-1$, the $r$-part of $(q^d-1)/(q-1)$ is $d_r$ by Lemma \ref{lem:2parts}, and the $r$-part of $(q^{d-1}-1)$ is  $(d-1)_r(q-1)_r$. For such tori $T$, we therefore have $|Q \cap T|/|T|  = 1/ d_r$ and $1/ \left((d-1)_r(q-1)_r\right)$ respectively. Since the proportion of $d$-cycles in $S_d$ is $1/d$ and the proportion of $(d-1)$-cycles is $1/(d-1)$, Theorem \ref{thm:maintool} implies 
 \begin{displaymath} 
 p_r(\SL_d(q)) \ge \frac{1}{d d_r} + \frac{1}{(d-1)_r(q-1)_r(d-1)}.
 \end{displaymath}
 Taking the same two classes of tori when $d \ge 3$ and $r=2$ implies that
 \begin{displaymath} 
 p_2(\SL_d(q)) \ge \begin{cases}
  \frac{2}{d d_2(q+1)_2} + \frac{1}{(d-1)(q-1)_2} & \mbox{when $d$ is even;}  \\
   \frac{1}{d} + \frac{2}{(d-1)(d-1)_2(q^2-1)_2}.& \mbox{when $d$ is odd.}
 \end{cases}
 \end{displaymath}
 
 Using Lemma \ref{QZ}, we find that the proportion of $r$-regular elements in $\PSL_d(q)$ is therefore
 \begin{displaymath} 
p_r(\PSL_d(q)) =    (d,q-1)_r p_r(\SL_d(q)).
 \end{displaymath}
 We claim that $p_r(\PSL_d(q)) \ge 1/d$ in all cases. 
  If $r$ is odd and $d_r=(d,q-1)_r$, then this is clear. On the other hand, if $r$ is odd and  $d_r > (d,q-1)_r = (q-1)_r$, then $(d-1)_r=1$ and therefore $p_r(\PSL_d(q)) \ge 1/(d-1)$. If $r=2$ and $d$ is odd, then the result is also clear.   If $r=2$ and $(d,q-1)_2=(q-1)_2$  (so $d$ is even), then it follows that $p_2(\PSL_d(q)) \ge 1/(d-1)$. If $(q-1)_2 > (d,q-1)_2=d_2$ and $d$ is even, then observe that $q \equiv 1 \imod{4}$ and $(q+1)_2=2$. Now the result follows quickly in this case as well. 
  It remains to prove the bounds for $\SL_2(q)$. The tori corresponding to $\{1\}$  in $S_2$ are cyclic of order $q-1$ and the tori corresponding to $\{(12)\}$  are cyclic of order $q+1$. Theorem \ref{thm:maintool} implies 
  $p_r(\SL_2(q)) \ge \frac{1}{2(q-1)_r} + \frac{1}{2(q+1)_r}$ as required. Finally 
  \[ p_r(\PSL_2(q)) = p_r(\SL_2(q)) (2,q-1)_r \ge \frac{1}{2}.\] 
   \end{proof}

\begin{lemma} \label{l:LUB} 
Let $\epsilon>0$, and let $p$ and $r$ be distinct primes. Then there exist infinitely many powers $q$ of $p$ for which $p_r(\PSL_d(q)) < 1/d + \epsilon$. In particular, part \emph{(ii)} of Theorem \ref{t:main} holds when $G=\PSL_d(q)$.
\end{lemma}
 \begin{proof} 
   Let $\epsilon>0$ and $ d\ge 2$,  let $p$ be a prime and let $r$ be a prime distinct from $p$. Choose a positive integer $a$ such that $1/r^a < \epsilon/d_r$ and $r^a \ge d_r$. Now choose  a positive integer $j$ such that $r^a | p^j-1$. Note that for any positive integer $b$, if we let $q = p^{jr^b}$, then $r^a | q-1$ by Lemma \ref{lem:2parts} and $(d,q-1)_r = d_r$. Now observe that for all classes in $W$, other than the class of $d$-cycles, the corresponding tori in $\SL_d(q)$ all have at least one cyclic factor of order $q^i-1$. By Lemma \ref{lem:2parts}, for any such torus $T$ we have 
 \[|Q \cap T|/|T| \le \frac{1}{(q^i-1)_r }\le \frac{1}{r^{a}} < \frac{\epsilon}{d_r}.\]
 For the cyclic tori $T$ of order $(q^d-1)/(q-1)$ we have $|Q \cap T|/ |T|  = 1/ d_r$. Applying Theorem \ref{thm:maintool}  and this inequality gives
 \[p_r(\SL_d(q)) =  \sum_{\B \subset W}\frac{|\B|}{|W|}.\frac{|T_{\B} \cap Q(r,\GF)|}{|T_{\B}|} \le   \frac{(d-1)\epsilon}{dd_r} + \frac{1}{dd_r} < \frac{1}{d_r} \left(\epsilon+d^{-1}\right).\] 
 Now applying Lemma \ref{QZ} together with the fact that $(d,q-1)_r = d_r$  implies  
 \[p_r(\PSL_d(q)) < \epsilon + d^{-1}\]
 as required. 
  \end{proof}
%
  \begin{lemma} \label{l:pslallq}
    For each odd prime power $q$ and  for all $d \ge 2$ we have
    \[p_2(\PSL_d(q)) \le  \frac{2(d,q-1)_2}{\sqrt{\pi d}}.\]
    In particular Theorem \ref{t:simpleub1} holds for these groups with $C_q= 2(q-1)_2/\sqrt{\pi}$.
  \end{lemma}
  \begin{proof} 
   Note that $p_2(\PSL_d(q)) = (d,q-1)_2 p_2 (\SL_d(q))$ by Lemma \ref{QZ} and set $Q = Q(2,\SL_d(q))$. Now observe that for all classes in $W$ with $k \ge 2$ cycles, the corresponding tori in $\SL_d(q)$ have $k-1$ cyclic factors of order $(q^{i}-1)$ for various $i$. By Lemma \ref{lem:2parts}, for any such torus $T$, we have
   \[ \frac{|Q \cap T|}{|T|} \le \frac{1}{2^{k-1}}.\]
Applying Theorem \ref{thm:maintool}  and this inequality gives
 \[p_2(\PSL_d(q)) \le  (d,q-1)_2 \left( \frac{1}{d} + \sum_{k =2}^{d} \frac{c(d,k)}{d!2^{k-1}}\right ) = 2(d,q-1)_2 \sum_{k =1}^{d} \frac{c(d,k)}{d!2^{k}} \]
 and the required bound follows from  \eqref{e:fr}. 
   \end{proof} 
   
 \begin{remark} 
  \emph{We note that Lemma \ref{l:pslallq} improves \cite[Theorem 6.1]{BPS}, which states that for $q \equiv 3	\pmod{4}$ such that $(d,q-1) \le 2$, the proportion of elements of odd order is  $p_2(\PSL_d(q)) \le 4/q+ 4/\sqrt{\pi d}$. }
 \end{remark}
 
 \section{Unitary case}
 
Suppose that $G =\SU_d(q)$, and since $\SU_2(q) \cong \SL_2(q)$, we may assume $d \ge 3$. Now $W =S_d$ and an $F$-class in $W$ consisting of cycles of length $b_1, \ldots , b_k$ corresponds to a $\GF$-class of maximal tori $T$, which are subgroups of 
\begin{equation}
 \prod_{i=1}^m (q^{b_i}-(-1)^{b_i})
\end{equation}
of index $q+1$. Again set $Q= Q(r,G)$.

\begin{lemma}  \label{l:psu1}
Let $m'$ be the multiplicative order of $q$ modulo $r$. Define 
  \[
m = \begin{cases}
1& \mbox{if $r=2$;}\\
 2m' & \mbox{if $m'$ is odd and $r \ne 2$;} \\
  m'  & \mbox{if $m' \equiv 0 \imod{4}$ and $r \ne 2$;}\\
  m'/2  & \mbox{if $m' \equiv 2 \imod{4}$ and $r \ne 2$.}
\end{cases}\]
  Then 
 \begin{align}  \label{e:su}
 p_r(\SU_d(q)) \ge \begin{cases}
 s_{\neg m}(d) \ge c(m)d^{-1/m} & \mbox{if } m \ge 2;\\ 
   \frac{1}{d d_r} + \frac{1}{(d-1)(d-1)_r (q+1)_r}  & \mbox{for $m=1$ and $r$ odd;} \\
    \frac{1}{d} + \frac{2}{(d-1)(d-1)_2 (q^2-1)_2}  & \mbox{for $r=2$ and $d$ odd;} \\
     \frac{2}{d d_2(q-1)_2} + \frac{1}{(d-1)(q+1)_2}  & \mbox{for $r=2$ and $d$ even.}  
\end{cases}
 \end{align}
and for $G = \PSU_d(q)$ the lower bound $p_r(G) \ge 1/d$ in part {\rm (i)} of Theorem \ref{t:main} holds.
  \end{lemma}
\begin{proof} 
  First we show that $m$ is the smallest positive integer for which $r | q^m-(-1)^{m}$. If $r=2$ then $r|q+1$ and $m=1$ so the claim is true. Suppose now that $r$ is odd. Note that $r| q^{i}-1$ if and only if $m' | i$. If $m'$ is odd then $r\nmid q^{i}+1$ for any $i$ so we seek the least even $i$ such that $r | q^{i}-1$; the least such $i$ is $2m'=m$. So suppose that $m'$ is even. Then $r | q^{m'/2}+1$ and $m'/2$ is the least integer $i$ such that $r | q^{i}+1$. If $m' \equiv 2 \imod{4}$ then $m=m'/2$ is odd, and hence $q^{m'/2}+1=q^{m}-(-1)^{m}$ and the claim is proved. This leaves $m' \equiv 0\imod{4}$ in which case $m=m'$, and there is no odd $i$ such that $r | q^{i}+1$. So the claim holds in this case also.

\par
 First suppose that $r$ is odd. If $m$ is even, and $r | (q^{b_i}-1)$ (with $b_i$ even)
 then $r$ divides $(q^{b_i}-1, q^m-1)$, which equals $q^{(b_i,m)}-1$ by Lemma \ref{l:nt}, and therefore $m | b_i$.  If $m$ is even, and $r | (q^{b_i}+1)$ (with $b_i$ odd),
 then $r$ divides $(q^{b_i}+1, q^m-1)$, which equals $q^{(b_i,m)}+1$ by Lemma \ref{l:nt} and therefore $m | b_i$. 
  It follows that if $m$ is even and $w \in S_n$ has no cycle of length divisible by $m$, then $r \nmid |T_w|$ and $|T_w \cap Q|/ |T_w| =1$. 
  
   Similarly, if $m \ge 3$ is odd, and $r | (q^{b_i}-1)$ (with $b_i$ even)
 then $r$ divides $(q^{b_i}-1, q^m+1)$, which equals $q^{(b_i,m)}+1$ and therefore $m | b_i$.  If $m$ is odd, and $r | (q^{b_i}+1)$ (with $b_i$ odd),
 then $r$ divides $(q^{b_i}+1, q^m+1)$, which equals $q^{(b_i,m)}+1$ by Lemma \ref{l:nt}, and therefore $m | b_i$. 
  It follows that if $m$ is odd and $m \ge 3$, and if $w \in S_n$ has no cycles of length divisible by $m$, then $r \nmid |T_w|$ and $|T_w \cap Q|/ |T_w| =1$. 
  
 Using Theorem \ref{thm:maintool} we find that
 \begin{displaymath} 
p_r(\SU_d(q))=  \frac{|Q|}{|\SU_d(q)|} \ge s_{\neg m}(d).
 \end{displaymath}
 As in the linear case, we can verify that $ s_{\neg m}(d) \ge 1/d$ in all cases. The proportion of $r$-regular elements in $\PSU_d(q)$  when $r$ is odd and $m\ge 2$  is therefore 
 \[p_r(\PSU_d(q)) \ge  (d,q+1)_r s_{\neg m}(d)  \ge d^{-1}\] 
  by Lemma \ref{QZ}.
  
  Now suppose $m=1$ so that $r |q+1$. As in the linear case, we consider the conjugacy classes of $d$-cycles and $(d-1)$-cycles in $S_d$, which correspond to the classes of cyclic maximal tori in $\SU_d(q)$ of orders $(q^d-(-1)^d)/(q+1)$ and $(q^{d-1}-(-1)^{d-1})$ respectively. Since $r |q+1$, when $r$ and $d$ are odd,  the $r$-part of $(q^d+1)/(q+1)$ is $d_r$ by Lemma \ref{lem:2parts}. Similarly if $r$ is odd and $d$ is even, then the $r$-part of $(q^d-1)/(q+1)$ is also $d_r$. And if $r$ is odd, then the $r$-part of $(q^{d-1}-(-1)^{d-1})$ is $(d-1)_r(q+1)_r$. Applying Theorem \ref{thm:maintool} proves the lemma in the case $m=1$ and $r$ odd. The cases with  $r=2$ follow from entirely similar applications of Lemma \ref{lem:2parts} and Theorem \ref{thm:maintool} and  we omit the details.  Thus, the inequality \eqref{e:su} is proved.
  
  By Lemma \ref{QZ}, we have $p_r(\PSU_d(q))  = (d,q+1)_rp_r(\SU_d(q))$ and the inequality $p_r(\PSU_d(q)) \ge d^{-1}$ follows by the same argument as in the linear case. 
    \end{proof}

\begin{lemma} \label{l:UUB}
  Let $\epsilon>0$, and let $p$ and $r$ be distinct primes such that the multiplicative order of $p$ modulo $r$ is even. Then there exist infinitely many prime powers $q$ of $p$ for which $p_r(\PSU_d(q)) < 1/d + \epsilon$. In particular, 
  part \emph{(ii)} of Theorem \ref{t:main} holds when $G=\PSU_d(q)$.
\end{lemma}
 \begin{proof} 
   Let $\epsilon>0$ and $d\ge 2$, and let $r$ be a prime. As in the linear case, choose a positive integer $a$ such that $1/r^a < \epsilon/d_r$ and $r^a \ge d_r$. Since the multiplicative order of $p$ modulo $r$ is even, there exists a positive integer $j'$ such that $r | p^{j'}+1$. In the light of Lemma \ref{lem:2parts}, there exists a positive integer $j$ such that $r^a | p^j+1$. Now for any positive integer $b$, if we let $q = p^{jr^b}$ for $r$ odd and $q=p^{j3^b}$ for $r=2$, then $r^a | q+1$ by Lemma \ref{lem:2parts} and $(d,q+1)_r = d_r$. We observe that for all classes in $W$, other than the class of $d$-cycles, the corresponding tori in $\SU_d(q)$ all have at least one cyclic factor of order $q^i-(-1)^{i}$. By Lemma \ref{lem:2parts}, for any such torus $T$ we have 
 \[ \frac{|Q \cap T|}{|T|} \le  \frac{1}{(q^i-(-1)^{i})_r} \le \frac{1}{r^{a}} < \frac{\epsilon}{d_r}.\]
 Similarly, for a cyclic torus $T$ of order $(q^d-(-1)^{d})/(q+1)$ we have $|Q \cap T|/ |T|  = 1/ d_r$ by Lemma \ref{lem:2parts} (note that if $r=2$, then $q \equiv 3 \imod{4}$). Applying Theorem \ref{thm:maintool}  and this inequality gives
 \[ p_r(\SU_d(q))  =  \sum_{B \subset W}\frac{|B|}{|W|}.\frac{|T_B \cap Q(r,\GF)|}{|T_B|} \le  \frac{(d-1)\epsilon}{dd_r} + \frac{1}{dd_r}.\] 
 Now applying Lemma \ref{QZ} together with the fact that $(d,q+1)_r = d_r$  implies  
 \[p_r(\PSU_d(q)) < \epsilon + d^{-1}\]
 as required. 
  \end{proof}
 \begin{lemma} \label{l:psuallq}
    For each prime power $q$ and for all $d \ge 2$ we have 
    \[p_2(\PSU_d(q)) \le  \frac{2(d,q+1)_2}{\sqrt{\pi d}}.\]
    In particular Theorem \ref{t:simpleub1} holds for these groups with $C_q = 2(q+1)_2/\sqrt{\pi}$.
  \end{lemma}
  \begin{proof} 
   Note that $p_2(\PSU_d(q)) = (d,q+1)_2 p_2 (\SU_d(q))$ by Lemma \ref{QZ} and set $Q = Q(2,\SU_d(q))$. Now observe that for all classes in $W$ with $k \ge 2$ cycles, the corresponding tori in $\SU_d(q)$ have $k-1$ cyclic factors of order $(q^{i}-(-1)^{i})$ for various $i$. By Lemma \ref{lem:2parts}, for any such torus $T$, we have
   \[ \frac{|Q \cap T|}{|T|} \le \frac{1}{2^{k-1}}.\]
Applying Theorem \ref{thm:maintool}  and the inequality above gives
 \[p_2(\PSU_d(q)) \le  (d,q+1)_2 \left( \frac{1}{d} + \sum_{k =2}^{d} \frac{c(d,k)}{d!2^{k-1}}\right ) \le 2(d,q+1)_2 \sum_{k =1}^{d} \frac{c(d,k)}{d!2^{k}} \]
 and the required bound follows from \eqref{e:fr}.
  \end{proof} 

\section{Symplectic and odd dimensional orthogonal cases}

Suppose that $\GF = \mathrm{Sp}_{2\l}(q)$. Then $W=C_2 \wr S_{\l} \le S_{2 \l}$, and a partition
\[\beta = (\beta^+, \beta^-)=(b_1^+,\ldots, b_{k_1}^+, b_1^-, \ldots, b_{k_2}^-)\]
 of $\l$, representing an $F$-class in $W$, corresponds to a $\GF$-class of maximal tori $T$ of the form
\begin{equation} \label{e:TcaseS}
T= \prod_{i=1}^{k_1} (q^{b_i^+}-1) \times \prod_{j=1}^{k_2} (q^{b_j^-}+1).
\end{equation}
For $w \in W=C_2 \wr S_{\l}$, write $\sigma(w)$ for the natural image of $w$ in $S_{\l}$. The partition $\beta$, without the signs, represents the cycle structure of $\sigma(w)$ in $S_{\l}$; the $+$ sign indicates that the $b_1^{+}$ cycle in $S_{\l}$ is the image of two cycles of length $b_1$ in $W\le S_{2\l}$, and these will be referred to as \textit{positive cycles}. The $-$ sign indicates that the $b_1^{-}$ cycle in $S_{\l}$ is the image of one cycle of length $2b_1$ in $W\le S_{2\l}$, and  these will be referred to as \textit{negative cycles}. Recall the definition of $\frfn$ in  \eqref{e:frfact} and that, by \eqref{e:fr}, we have
\[ \fr =  \sum_{k=1}^n \frac{c(n,k)}{n! 2^k(2,r)^k}. \]
 If $\GF = \SO_{2\l+1}(q)$, then the Weyl group $W$ and the correspondence between classes in $W$ and $\GF$-classes of maximal tori are exactly the same as in the case $\GF=\Sp_{2\l}(q)$.

\begin{lemma}  \label{l:Sp}
 Let $G = \Sp_{2n}(q)$, $\SO_{2n+1}(q)$ or $\Om_{2n+1}(q)$ where $ n\ge 2$.  If $r$ is a prime not dividing $q$ and $m$ is the multiplicative order of $q$ modulo $r$, then 
 \[ p_r(G) \ge 
 \begin{cases}
 s_{\neg m}(n)  & \hbox{if $m$ is odd and $m \ge 3$;} \\ 
  s_{\neg \frac{m}{2}}(n)   & \hbox{if $m$ is even and $m \ge 4$;} \\
  \fr & \hbox{otherwise.}  \\
      \end{cases}
     \]
 Also, $p_2(\Om_{2n+1}(q)) = 2p_2(\SO_{2n+1}(q))$ and $p_r(\PSp_{2n}(q)) = (2,r)p_r(\Sp_{2n}(q))$, and parts \emph{(i)} and \emph{(iii)} of Theorem \ref{t:main} hold for these groups. 
 \end{lemma}
\begin{remark} 
 \emph{Since $-I_{2n+1} \in Z(\GO_{2n+1}(q))$ has determinant $-1$, it follows that   $\PSO_{2n+1}(q)=\SO_{2n+1}(q)$ and  $\POm_{2n+1}(q) =\Om_{2n+1}(q)$.
} 
\end{remark} 
\begin{proof} 
 First let us assume that $\GF= \mathrm{Sp}_{2\l}(q)$. We may assume that $r$ is odd since the case $r=2$ is part of Proposition 6.2 in \cite{GP}.
Observe that $r | q^i-1$ if and only if $i$ is a multiple of $m$.  Similarly, $r | q^i+1$ if and only if $q^i \equiv -1 \imod{r}$ if and only if $m$ is even and $i$ is an odd multiple of $m/2$. It follows that if $m \ge 3$ is odd, and $\sigma(w) \in S_n$ has no cycle lengths $b_i$ divisible by $m$, then $r \nmid (q^{b_i}-1)$ and $r \nmid (q^{b_i}+1)$. In particular, $r \nmid |T_w|$ for all such $w \in W$. Applying Theorem \ref{thm:maintool} with $C$ the union of  classes in $W$ satisfying this condition implies
$p_r(\Sp_{2n}(q)) \ge s_{\neg m}(n)$
when $m \ge 3$ is odd. Similarly, if $m \ge 4$ is even and $\sigma(w) \in S_n$ has no cycle of length divisible by $m/2$, then $r \nmid |T_w|$ and 
$p_r(\Sp_{2n}(q)) \ge s_{\neg \frac{m}{2}}(n)$. 
\par

For each positive integer $i$, we have $\gcd(q^i-1,q^i+1)=(2,q-1)$. Since $r$ is odd, it follows that $r$ cannot divide both $q^i-1$ and $q^i+1$. Therefore  for all $\tau \in S_n$ it is possible to label each of the cycles of an arbitrary element $\tau \in S_n$ as positive or negative to give $w \in W$ satisfying $\sigma(w) = \tau$ and $r \nmid |T_w|$. We apply Theorem~\ref{thm:maintool} with $C$ the union of classes $B$ such that $r \nmid |T_B|$. If $\tau$ has $k$ cycles, then at least $2^{n-k}$ of the $2^n$ elements in $\{ w \in W \mid \sigma(w) =\tau \}$ satisfy the required condition. Therefore if $c(\l,k)$ denotes the number of permutations of $S_{\l}$ with exactly $k$ cycles, then we have 
\[
 p_r(\Sp_{2n}(q)) \geq \sum_{k=1}^{\l}\frac{2^{\l-k}c(n,k)}{|W|}
= \sum_{k=1}^{\l}\frac{c(n,k)}{2^k{\l}!} = \fr.
\]
Now we apply Lemma \ref{QZ} (noting that $(2,r)=1$) to  obtain
\[ p_r(\PSp_{2n}(q))  = |Z|_r p_r(\Sp_{2n}(q)) = (2,r)p_r(\Sp_{2n}(q)). \]
The same calculations prove the results for $\SO_{2\l+1}(q)$ and $\PSO_{2\l+1}(q)$. 
If $G = \Om_{2n+1}(q)$ and $r$ is odd then it is easy to see that 
\[Q_{\Om}:= Q(r, \SO_{2n+1}(q)) \cap \Om_{2n+1}(q)\]
 is also a quokka set. Moreover for each class $B$ chosen above, we have $T_B \subset Q(r, \SO_{2n+1}(q))$ and thus $|T_B \cap Q_{\Om}| = |T_B \cap \Om_{2n+1}(q)| =  |T_B|/2$ (see \cite[Theorem 4]{BuGr} for example). Theorem \ref{thm:maintool} therefore implies
\[ \frac{|Q_\Om|}{|\SO_{2n+1}(q)|}  \ge  \frac{(2n)!}{2^{2n+1}{\l}!^2}. \] 
It now follows easily that if $r$ is odd, then
\[p_r( \Om_{2n+1}(q) ) =  \frac{|Q_\Om|}{|\Om_{2n+1}(q)|} \ge \frac{(2n)!}{2^{2n}{\l}!^2}= \fr.\]

Finally, observe that all $2$-regular elements in $\SO_{2n+1}(q)$ are actually contained in $\Om_{2n+1}(q)$; thus $p_2(\Om_{2n+1}(q)) = 2p_2(\SO_{2n+1}(q))$.
\end{proof}

\begin{lemma}  \label{l:SpUB}
  Let $\epsilon>0$, $n\ge2$ an integer, and let $p$ and $r$ be distinct primes. Then there exist infinitely many powers $q$ of $p$ for which  $p_r(\Om_{2n+1}(q))$ is less than $(2,r)\fr + \epsilon$, and  $p_r(\Sp_{2n}(q))$ and $p_r(\SO_{2n+1}(q))$ are less than $\fr+ \epsilon$.
  In particular, part \emph{(ii)} of Theorem \ref{t:main} holds for these groups. 
  \end{lemma}
 \begin{proof} 
     Choose a positive integer $a$ such that $1/r^a < \epsilon/2$. Now choose  a positive integer $j$ such that $r^a | p^j-1$. Note that for any positive integer $b$, if we let $q = p^{jr^b}$, then $r^a | q-1$ by Lemma \ref{lem:2parts}. We show that $p_r(G) < \fr + \epsilon/2$ for $G=\Sp_{2n}(q)$ and $\SO_{2n+1}(q)$ so that the result for $p_2(\Om_{2n+1}(q))$ follows from the equation $p_2(\Om_{2n+1}(q)) = 2p_2(\SO_{2n+1}(q))$ obtained in Lemma \ref{l:Sp}. Now observe that for all classes in $W$ with at least one positive cycle, that is, classes whose corresponding tori have at least one cyclic factor of order $q^i-1$, we have  
 \[ \frac{|Q \cap T|}{|T|} \le \frac{1}{ (q^i-1)_r }\le \frac{1}{r^{a}}< \frac{\epsilon}{2}.\]
 Arguing as in the previous proof, for each of the $c(n,k)$ permutations $\sigma(w)$ of $S_n$ with exactly $k$ cycles, exactly $2^{n-k}$ of the corresponding $2^n$ elements $w$ of $W$ have $k$ negative cycles.  For each torus $T$ with $k$ negative cycles and no positive cycles,  we have $|Q \cap T|/|T | = 1/(2,r)^k$. Applying this inequality and  Theorem \ref{thm:maintool} for $G = \Sp_{2n}(q)$ or $\SO_{2n+1}(q)$  gives 
 \[p_r(G) =  \sum_{B \subset W}\frac{|B|}{|W|}.\frac{|T_B \cap Q(r,\GF)|}{|T_B|} <  \frac{\epsilon}{2} + \sum^n_{k=1} \frac{c(n,k)}{2^k (2,r)^k n!} = \frac{\epsilon}{2} + \fr. \] 
 The bound for $G= \Om_{2n+1}(q)$  follows by the same argument as in Lemma \ref{l:Sp} for $r$ odd and from the equation  $p_2(\Om_{2n+1}(q)) = 2p_2(\SO_{2n+1}(q))$ for $r=2$.
 \end{proof}
 \begin{corollary} \label{c:S}
 Let $G=\Om_{2n+1}(q)$, $\SO_{2n+1}(q)$ or $\Sp_{2n}(q)$, and let $r$ be a prime not dividing $q$. Then 
 \[ p_r(G) \ge \begin{cases}
  \frac{25}{29\sqrt{\pi n}} & \hbox{for $r$ odd;}\\
  \frac{1}{4(n+1)^{3/4}} & \hbox{for $r=2$ and $G \ne \Om_{2n+1}(q)$;}\\
  \frac{1}{2(n+1)^{3/4}} & \hbox{for $r=2$ and $G = \Om_{2n+1}(q)$.}
 \end{cases}
 \]
 Moreover, for any prime $p$, there exist infinitely many powers $q$ of $p$ for which 
  \[ p_r(G) \le \begin{cases}
  \frac{1}{\sqrt{\pi n}} & \hbox{for $r$ odd;}\\
  \frac{3}{5(n+1)^{3/4}} & \hbox{for $r=2$ and $G \ne \Om_{2n+1}(q)$;}\\
  \frac{6}{5(n+1)^{3/4}} & \hbox{for $r=2$ and $G = \Om_{2n+1}(q)$.}
 \end{cases}
  \]
  Furthermore there exist constants $C'_5$ and $C'_6$ such that
  \[ p_2(G) \le \begin{cases}
  \frac{C'_5}{n^{5/8}} & \hbox{if $q \equiv 1\pmod{4}$;}\\
  \frac{C'_6}{n^{1/2}} & \hbox{if $q \equiv -1\pmod{4}$;}
   \end{cases}
  \]
Therefore Theorems \ref{t:simplelb}, \ref{t:simplebestlb} and \ref{t:simpleub1} hold for these groups.
 \end{corollary}
 \begin{proof} 
  The first two sets of inequalities follow easily from (4.14) in \cite{GP} and our bounds in Lemma \ref{l:npi}. For the final set of inequalities first suppose $q \equiv 1\pmod{4}$ and set $Q=Q(2,G)$. Observe that for each positive cycle, we have a cyclic factor $(q^{i}-1)$ in the corresponding torus and $(q^{i}-1)_2 \ge 4$. Similarly, for each negative cycle, we have a cyclic factor $(q^{i}+1)$ in the corresponding torus and $(q^{i}+1)_2 = 2$. So for $w \in W$ with $i$ positive cycles and $k-i$ negative cycles, the proportion of $2$-regular elements in $T$ is 
  \[ \frac{|Q \cap T|}{|T|} \le \frac{1}{4^{i} 2^{k-i}} = \frac{1}{2^{k+i}}.\]
  We note that precisely $c(n,k) 2^{n-k}{k \choose i}$ of the elements in $W$ have $i$ positive cycles and $k-i$ negative cycles and therefore Theorem \ref{thm:maintool} gives 
  \[ p_2(\Sp_d(q)) \le \sum_{k=1}^{n}\frac{c(n,k)}{n! 2^{k}} \sum_{i=0}^{k} {k \choose i} \frac{1}{2^{k+i}}=\sum_{k=1}^{n}\frac{c(n,k)}{n! } \left ( \frac{3}{8}\right )^{k}.\]

  Using \eqref{e:sumcnkbound}, it follows that $p_2(G) \le C'_5 d^{-5/8}$ for some constant $C'_5$. A similar calculation for $q \equiv 3\pmod{4}$ with  
     \[ \frac{|Q \cap T|}{|T|} \le \frac{1}{ 2^{k}}\] 
     yields the upper bound $f_2(n)$ and, by Lemma \ref{l:npi}, this is at most $C'_6 d^{-1/2}$ for some constant $C'_6$.
 \end{proof}

%

%

\section{Even dimensional orthogonal cases}

The structure of the maximal tori in $\SO^{\pm}_{2\l}(q)$ is essentially the same as in the symplectic case, but the Weyl group has index 2 in the Weyl group $W_{B_n}:=C_2 \wr S_{\l}$ for $B_n$. Moreover, in the untwisted case, an $F$-conjugacy class in $W$ is a conjugacy class of permutations in $W_{B_n}$ with an even number of negative cycles (that is, $k_2$ in \eqref{e:TcaseS} is even). In the twisted case, an $F$-conjugacy class $C$ in $W$ corresponds to a conjugacy class $C^{\prime}$ of permutations in $W_{B_n}$ with an odd number of negative cycles (that is, $k_2$ is odd); if we define the $2$-cycle $w=(\l \, \l^{\prime}) \in W_{B_n}$, then $C=wC^{\prime}$.
Recall that  $\goddfn$ and $\gevenfn$ are defined in \eqref{e:goddgamma} and \eqref{e:gevengamma} in terms of the Gamma function and that we also have from \eqref{e:godd} and \eqref{e:geven} that 
\begin{align*} 
 \godd = 2\sum_{\substack{{k=1}\\{k \text{ odd}}}}^{n} \frac{c(n,k)}{n!2^k(2,r)^k} & \hspace{0.5cm}  \hbox{ and}&  \geven = 2\sum_{\substack{{k=1}\\{k \text{ even}}}}^{n} \frac{c(n,k)}{n!2^k(2,r)^k}.
\end{align*}
 The \emph{Witt defect} of $\SO_{2n}^{\pm}(q)$ or $\Om_{2n}^{\pm}(q)$ is $0$ if the group is of $+$ type and $1$ if the group is of $-$ type.

\begin{lemma} \label{l:SO}
 Let $G=\SO_{2n}^{\pm}(q)$ or $\Om_{2n}^{\pm}(q)$ where $ n\ge 2$. Let $m$ be the multiplicative order of $q$ modulo $r$ and let $l$ be the Witt defect of $G$. If $r$ is odd, then
\begin{align} 
 p_r(G)  \ge  \begin{cases}
 s_{\neg m}(n)  & \hbox{if $m$ is odd and $m \ge 3$;} \\ 
  s_{\neg \frac{m}{2}}(n)   & \hbox{if $m$ is even and $m \ge 4$;} \\
   g_{odd,r}(n) & \hbox{if $m=2$ and $n+l$ is odd, or if $l=m=1$;}  \\
     g_{even,r}(n) & \hbox{if $m=2$, $n+l$ is even, or if $m=1$ and $l=0$.}  
      \end{cases}  
\end{align} 
And 
\begin{align} 
 p_2(\SO^{\pm}_{2n}(q))  \ge  \begin{cases}
   \geven & \hbox{if $q^n \equiv \pm 1 \imod{4}$;}  \\
   \godd   & \hbox{if  $q^n \equiv \mp 1 \imod{4}$.}  
      \end{cases}  
\end{align} 
   Also $p_r(G/Z(G)) = |Z(G)|_rp_r(G)$ and $p_2(\Om^{\pm}_{2n}(q)) = 2 p_2(\SO^{\pm}_{2n}(q))$. Parts \emph{(i)} and \emph{(iii)} of Theorem \ref{t:main} hold for these groups.
\end{lemma}
\begin{remark} 
 \emph{In this case, $Z(\SO^{\pm}_{2n}(q)) = \langle -I_{2n} \rangle$ since $-I_{2n}$ has determinant $1$. The centre of $\Om^{\pm}_{2n}(q)$ is nontrivial if and only if $q$ is odd and  $-I_{2n}$ has spinor norm $1$. It turns out that  $Z(\Om^{+}_{2n}(q)) = \langle -I_{2n} \rangle$ if and only if $q^{n} \equiv 1 \imod{4}$, and is trivial otherwise. Similarly, $Z(\Om^{-}_{2n}(q)) = \langle -I_{2n} \rangle$ if and only if $q^{n} \equiv -1 \imod{4}$, and is trivial otherwise.}
\end{remark}
\begin{proof} 
First of all, we suppose that $G= \SO^{\pm}_{2\l}(q)$. Again we may assume $r$ is odd by Proposition 7.1 of \cite{GP}. Observe that $r | q^i-1$ if and only if $i$ is a multiple of $m$.  Similarly, $r | q^i+1$ if and only if $q^i \equiv -1 \imod{r}$ if and only if $m$ is even and $i$ is an odd multiple of $m/2$. It follows that if $m \ge 3$ is odd, and $\sigma(w) \in S_n$ has no cycle lengths $b_i$ divisible by $m$, then $r \nmid (q^{b_i}-1)$ and $r \nmid (q^{b_i}+1)$. In particular, $r \nmid |T_w|$ for all such $w \in W$. Applying Theorem \ref{thm:maintool} with $C$ the union of  classes in $W$ satisfying this condition implies
$p_r(\SO_{2n}^{\pm}(q)) \ge s_{\neg m}(n)$
when $m \ge 3$ is odd. Similarly, if $m \ge 4$ is even and $\sigma(w) \in S_n$ has no cycle of length divisible by $m/2$, then $r \nmid |T_w|$ and 
$p_r(\SO_{2n}^{\pm}(q)) \ge s_{\neg \frac{m}{2}}(n)$. 
\par
Now suppose that $m=1$ so that $r |q-1$. Then $r \nmid q^i+1$ for all positive integers $i$ and therefore if $C$ is the union of  classes $B$ in $W$ whose cycles are all negative, then $r \nmid |T_B|$ for all $B$. If the Witt defect $l=0$, then there must be an even number of negative cycles and thus $\sigma(w)$ must have an even number of cycles when $w \in C$. By the same argument as in the symplectic case, if $\tau \in S_n$ has $k$ cycles, then $2^{n-k}$ of the $2^{n}$ elements $w$ in $W_{B_n}$ satisfying $\sigma(w)=\tau$ have all cycles negative, and all of these elements are contained in $W$.
Applying Theorem \ref{thm:maintool} we have 
\[p_r(\SO^{+}_{2n}(q)) \ge \frac{|C|}{|W|} = \sum_{\substack{{k=1}\\{k \text{ even}}}}^{n} \frac{c(n,k)2^{n-k}}{n!2^{n-1}} = 2\sum_{\substack{{k=1}\\{k \text{ even}}}}^{n} \frac{c(n,k)}{n!2^{k}} = \geven \]
when $m=1$ and $l=0$.
If $m=l=1$,  then in order for $r$ not to divide $|T_w|$, the permutation $\sigma(w)$ must have an odd number of cycles (all of which are negative) and a similar calculation shows 
\[p_r(\SO^{-}_{2n}(q)) \ge 2 \sum_{\substack{{k=1}\\{k \text{ odd}}}}^{n} \frac{c(n,k)}{n!2^{k}} = \godd. \]
It remains to deal with the case $m=2$ (so $r|q+1$). Here $r | q^{i}+1$ if and only if $i$ is odd, and $r | q^{i}-1$ if and only if $i$ is even. Therefore we let $C$ be the union of classes in $W$ whose elements only have positive cycles of odd length and negative cycles of even length. If $n$ is even and $l=0$, then we claim that $w \in W$ can be contained in $C$ only if the number $k$ of cycles of $\sigma(w)$ is even. For if $n$ is even, then the number $\kodd$ of cycles  of $\sigma(w)$ of odd length must be even.
But the number of negative cycles of $w$ must also be even since $l=0$, and this is the same as the number $\keven$ of  cycles of even length. Thus $\kodd$ and $\keven$ must both be even and $k = \kodd + \keven$ must be even. This argument shows that if $n$ is even, then $\keven$ is even if and only if $k$ is even. By the same argument as before, if $\sigma(w)$ has $k$ cycles and $k$ is even, then precisely $2^{n-k}$ of the $2^{n}$ elements in $W_{B_n}$ have all of the cycles of the correct signs, and all of these elements are contained in $W$. Thus when $n$ is even, $m=1$,  and $l=0$, we have
\[p_r(\SO^{+}_{2n}(q)) \ge 2 \sum_{\substack{{k=1}\\{k \text{ even}}}}^{n} \frac{c(n,k)}{n!2^{k}}=\geven. \]
Similarly, if $n$ is odd, $m=1$ and $l=0$, then $w \in W$ can be contained in $C$ only if the number $k$ of cycles of $\sigma(w)$ is odd and the same argument gives
\[p_r(\SO^{+}_{2n}(q)) \ge 2 \sum_{\substack{{k=1}\\{k \text{ odd}}}}^{n} \frac{c(n,k)}{n!2^{k}}= \godd. \]
If $n$ is even, $m=1$ and $l=1$, then $k$ must be  odd. And if $n$ is odd, $m=1$ and $l=1$, then $k$ must be even to satisfy our condition. Theorem \ref{thm:maintool} implies the result in these last two cases.

The theorem for $G = \Om_{2n}^{\pm}(q)$ and $r$ odd is obtained in exactly the same way as in the odd dimensional orthogonal case. When $r=2$, we observe as before that all odd order elements in $ \SO_{2n}^{\pm}(q)$ are contained in  $\Om_{2n}^{\pm}(q)$ and therefore $p_2( \Om_{2n}^{\pm}(q)) = 2p_2( \SO_{2n}^{\pm}(q))$.

As usual, Lemma \ref{QZ} implies that $p_r(G/Z(G)) = |Z(G)|_r p_r(G)$ for all of these groups. 
\end{proof}

Recall from Lemma \ref{l:goddineq} that for all $ m\ge2$, we have $ s_{\neg m}(n) \ge \godd \ge \geven$ and notice that Lemma \ref{l:SO} proves that $p_r(\SO^+_{2n}(q)) \ge \geven$ for all $n \ge 2$, $p_r(\SO^-_{2n}(q)) \ge \godd$ if $n$ is even, and $p_r(\SO^-_{2n}(q)) \ge \geven$ if $ n\ge3$ is odd. Lemma \ref{l:SOUB} below shows that these are the best possible lower bounds that are independent of $q$. Here we denote the multiplicative order of $p$ modulo $r$ by  $o(p \imod{r})$.

\begin{lemma} \label{l:SOUB}
Let $G=X_{2n}(q)$ be an even dimensional orthogonal group defined over a field of characteristic $p$, and let $r$ be a prime distinct from $p$. Suppose that $X$, $n$, $r$, and $p$ satisfy the conditions in one of the rows of Table \ref{tab:caseO} and let $\epsilon>0$. Then there exist infinitely many powers $q$ of $p$ for which 
\begin{equation} \label{e:SOUB}
  p_r(G) < \hxr(n) +\epsilon.
\end{equation}
In particular, part \emph{(ii)} of Theorem \ref{t:main} holds for these groups.
\begin{table}[htdp]
\begin{center}\begin{tabular}{ccccc}
\hline  $X_d$ & {\rm Conditions} & $\hxr(n)$ & $o(p \imod{r})$ \vspace*{0.05cm} \\
\hline \\ [-2.5ex]
  $\SO_{2n}^+$& - & $\geven$ & {\rm any} \\
   $\Om_{2n}^+$& - & $(2,r)\geven$ & {\rm any} \\
 $\SO_{2n}^-$ & & $\godd$ & {\rm any} \\
  $\SO_{2n}^-$ & $n$ {\rm odd} & $\geven$ & {\rm even} \\
  $\Om_{2n}^-$&  & $(2,r)\godd$ & {\rm any} \\
  $\Om_{2n}^-$ & $n$ {\rm odd} & $(2,r)\geven$ & {\rm even} \vspace*{0.05cm}\\
  \hline 
  \end{tabular}
  \vspace*{0.15cm}
   \caption{Multiplicative orders of $p$ modulo $r$ for which there exist infinitely many powers $q$ of $p$ such that \eqref{e:SOUB} holds.} \label{tab:caseO}
\end{center}
\end{table}
  \end{lemma}
\begin{proof} 
  Set $Q=Q(r,G)$ and choose a positive integer $a$ such that $1/r^a < \epsilon/2$. Now choose  a positive integer $j$ such that $r^a | p^j-1$. Note that for any positive integer $b$, if we let $q = p^{jr^b}$, then $r^a | q-1$ by Lemma \ref{lem:2parts}. As in Lemma \ref{l:SpUB}, we show that $p_r(\SO^{\pm}_{2n}(q))< \hxr(n) + \epsilon/2$ and hence obtain the bound for $p_2(\Om_{2n}^{\pm}(q))$ from the equation $p_2(\Om^{\pm}_{2n}(q)) = 2p_2(\SO_{2n}^{\pm}(q))$ in Lemma \ref{l:SO}. Now observe that for all classes in $W$ whose corresponding tori have at least one cyclic factor of order $q^i-1$, we have  
 \[ \frac{|Q \cap T|}{|T| }\le \frac{1}{(q^i-1)_r} \le \frac{1}{r^{a}} < \frac{\epsilon}{2}.\]
 For a torus $T$ with precisely $k$ negative cycles,  we have $|Q \cap T|/|T | = 1/(2,r)^k$. Therefore if $l$ denotes the Witt defect of $G$, then  Theorem \ref{thm:maintool}  implies
 \[p_r(\SO_{2n}^{\pm}(q)) < \begin{cases}
  \epsilon/2 +  \geven&  \hbox{for $l=0$;}\\
  \epsilon/2 + \godd &  \hbox{for $l=1$.}
 \end{cases}\] 
 In the light of Lemma \ref{QZ}, this proves the lemma under the conditions of lines 1 and 3 of Table \ref{tab:caseO}. To prove it for line 4, suppose that the Witt defect $l=1$ and $ n\ge 3$ is odd. Let $\epsilon >0$ and choose an integer $a$  such that $ r^{-a} < \epsilon/2$. If the multiplicative order of $p$ modulo $r$ is even, then there exists a positive integer $j'$ such that $r | p^{j'}+1$. In the light of Lemma \ref{lem:2parts}, there exists a positive integer $j$ such that $r^a | p^j+1$. Now for any positive integer $b$, if we let $q = p^{jr^b}$ for $r$ odd and $q=p^{j3^b}$ for $r=2$, then $r^a | q+1$ by Lemma \ref{lem:2parts}. Moreover, $(r,q^{i}-1) = (2,r)$ if and only if $i$ is odd, and  $(r,q^{i}+1) = (2,r)$ if and only if $i$ is even. Therefore for all $w \in W$ with all odd length cycles positive and all even length cycles negative, we have $|Q\cap T_w|/|T_w| = 1/(2,r)^k$, where $k$ is the number of cycles of $\sigma(w) \in S_n$. Since $n$ is odd,  the number $\kodd$ of cycles of odd length of $\sigma(w) \in S_n$ is odd. If $w$ satisfies our condition then $\kodd=k_1$, the number of positive cycles is odd, and thus $\keven = k_2$, the number of negative cycles, is odd if and only if $k$ is even. 
  For all other $ w\in W$, $T_w$ has at least one cyclic factor of order $q^{i}-1$ with $i$ even, or $q^{i}+1$ with $i$ odd; thus $|Q\cap T_w|/|T_w| \le r^{-a}< \epsilon/2$.
  Applying Theorem \ref{thm:maintool}  gives
  \[ p_r(\SO^{-}_{2n}(q)) < \epsilon/2 + \geven,\]
as required.
 As usual, the same argument as the one in Lemma \ref{l:Sp} proves \eqref{e:SOUB} under the conditions of lines 2, 5, or 6 of Table \ref{tab:caseO} when $r$ is odd. Finally, the  equation $p_2(\Om^{\pm}_{2n}(q)) = 2p_2(\SO_{2n}^{\pm}(q))$ proves \eqref{e:SOUB} under the conditions of lines 2, 5, or 6 when $r=2$.
  \end{proof}
  
  \begin{corollary} \label{c:caseO}
  Let $r$ and $p$ be distinct primes with $q$ a power of $p$.
 Let $G=\Om_{2n}^{\pm}(q)$ or $\SO^{\pm}_{2n}(q)$. Then 
 \[ p_r(G) \ge \begin{cases}
 \frac{25}{29\sqrt{\pi n}}  \cdot \left( \frac{2n-2}{2n-1}\right )& \hbox{for $r$ odd}\\
  \frac{1}{8(n+1)^{3/4}}  -   \frac{9}{25(n+1)^{5/4}}& \hbox{for $r=2$ and $G \ne \Om_{2n}^{\pm}(q)$;}\\
   \frac{1}{4(n+1)^{3/4}}  -   \frac{18}{25(n+1)^{5/4}}& \hbox{for $r=2$ and $G = \Om_{2n}^{\pm}(q)$.}\\
 \end{cases}
 \]
 Moreover, given any prime $p$, there exist infinitely many powers $q$ of $p$ for which 
  \[ p_r(G) \le \begin{cases}
  \frac{1}{\sqrt{\pi n}}  \cdot \left( \frac{2n}{2n-1}\right ) & \hbox{for $r$ odd}\\
  \frac{3}{10(n+1)^{3/4}}  +   \frac{9}{25(n+1)^{5/4}} & \hbox{for $r=2$ and $G \ne \Om_{2n}^{\pm}(q)$;}\\
  \frac{3}{5(n+1)^{3/4}}  +   \frac{18}{25(n+1)^{5/4}} & \hbox{for $r=2$ and $G = \Om_{2n}^{\pm}(q)$.}
 \end{cases}
  \]
  Furthermore, there exist constants $C''_5$, $C''_6$ such that
   \[ p_2(G) \le \begin{cases}
  \frac{C''_5}{n^{5/8}} & \hbox{if $q \equiv 1\pmod{4}$;}\\
  \frac{C''_6}{n^{1/2}} & \hbox{if $q \equiv -1\pmod{4}$;}
   \end{cases} 
  \]
 Therefore Theorems \ref{t:simplelb}, \ref{t:simplebestlb} and \ref{t:simpleub1} hold for these groups.
 \end{corollary}
 \begin{proof} 
  The first two sets of inequalities follow easily from (4.16) and (4.17) in \cite{GP}, Lemma \ref{l:npi}, \eqref{e:goddgamma} and \eqref{e:gevengamma}. For the third set of inequalities, we argue in the same way as in Corollary \ref{c:S}.
 \end{proof}
 
 \section{Answers to some open questions} \label{openqs}
 
 In \cite[Section 8]{BPS}, the first author, P\'alfy and Saxl introduce the notation $\rho(r,X,d,q)$ to denote the proportion of $r$-regular elements in a classical simple group $X_d(q)$. One of their main results is that $\rho(r,X,d,q)  > 1/2d$ for all $r$, $X$ and $q$. Moreover if $X_d(q) = \PSL_d(q)$, they show that for infinitely many values of $q$ the bound $\rho(r,X,d,q)  \le 3/\sqrt{d}$ holds. They asked whether it is possible to close the quadratic gap. More specifically, they let
 \[ \alpha(r,X,q) = \limsup_{d \to \infty} \frac{ - \log \rho(r,X,d,q)  }{\log d} \] 
 and $\alpha = \sup_{r,X,q} \alpha(r,X,q)$. They observe that $1/2 \le \alpha \le 1$, but ask for the 
specific value of $\alpha$. This follows from Theorem \ref{t:limsup}. We combine the proof of 
Theorem~\ref{t:limsup} with our determination of $\alpha$ in Lemma~\ref{l:alpha}.
 \begin{lemma} \label{l:alpha}
 Let $\epsilon > 0$. Then there exists  $q$ and a constant $C$ such that $p_2(\PSL_d(q)) \le Cd^{\epsilon-1}$ for all $d \ge 2$. In particular, Theorem \ref{t:limsup} holds and 
  the supremum $\alpha = \sup_{r,X,q} \alpha(r,X,q)$ is equal to $1$. 
 \end{lemma}
 \begin{proof}
 First, choose a positive integer $a$ such that $2^{-a} < \epsilon$ and choose $q$ such that $2^{a} | q-1$. 
Then for every $w \in W$ with $k \ge 2$ cycles, there are, in the corresponding tori of $\SL_d(q)$, at least 
$k-1$ cyclic factors of order $q^{i}-1$ for various $i$. In particular the proportion of $2$-regular elements 
in such a torus $T$ is 
 \[ \frac{|Q \cap T|}{|T|} \le \frac{1}{2^{a(k-1)}}.\]
 Applying Theorem \ref{thm:maintool} implies
 \[ p_2(\SL_d(q)) \le \frac{1}{d} + \sum_{k=2}^{d} \frac{c(d,k)}{d! 2^{a(k-1)}} =  2^{a}\sum_{k=1}^{d} \frac{c(d,k)}{d! 2^{ak}} \]
 and 
 \[
\sum_{k=1}^{d} \frac{c(d,k)}{d! 2^{ak}} \le  \frac{C'}{d^{1-2^{-a}}}
\]
  by \eqref{e:sumcnkbound}, where $C'$ depends only on $\epsilon$. Now 
$p_2(\PSL_d(q))  = (d,q-1)_2p_2(\SL_d(q))$ and so we have shown that there 
exists $q$ and a constant $C$ (depending only on $\epsilon$ and $q$) such that
 \[p_2(\PSL_d(q))  \le  C/ d^{1-2^{-a}} \le C/ d^{1-\epsilon}\]
for all $d$. This proves Theorem~\ref{t:limsup}. 
  It follows that for any $\epsilon >0$, there exists $q$ such that $\alpha (2,\PSL,q)   \ge 1- \epsilon$ and therefore $\alpha = \sup_{r,X,q} \alpha(r,X,q) =1$.
\end{proof}
     
   The quantity $\alpha(r,X) = \inf_q \alpha(r,X,q)$ is also defined in \cite{BPS} and it is observed that $\alpha(r,\PSL) \ge 1/r$. They ask whether $\inf_r  \alpha(r,\PSL) =0$. We can prove that this is the case; indeed, we claim that for all $r$ we have $1/(r-1) \ge \alpha(r,\PSL) \ge 1/r$. If $r=2$ then there is nothing to prove so suppose $ r \ge 3$. We fix a prime $q$ with multiplicative order $r-1$ modulo $r$ (the existence of $q$ is guaranteed by the Dirichlet prime number theorem, see \cite[Theorem 7.9]{Ap} for example). By Remark \ref{rem:main}(c), we have $p_r(\SL_d(q)) \ge s_{\neg r-1}(d)$ and so \eqref{beals} implies that there exists a consant $C$ such that $p_r(\PSL_d(q)) \ge C d^{- \frac{1}{r-1}}$ for all $ d \ge 2$. It follows that  
   $\alpha(r,\PSL,q) \le \frac{1}{r-1}$ and hence $\alpha(r,\PSL) = \inf_q \alpha(r,\PSL,q) \le \frac{1}{r-1}$ as required. We also note that $\inf_r \alpha(r,X) =0$ for the other classes of simple group $X$ by Lemmas \ref{l:psu1}, \ref{l:Sp} and \ref{l:SO}.
   
   The authors of \cite{BPS} also state that they expect that $\alpha(2,X) \ge 1/2$ for the other types of classical simple group $X$. This is indeed the case since Theorem \ref{t:simpleub1} shows that for every $q$ and each $X$,  there exists a constant $C_q$ (which may depend on $q$) such that $\rho(2,X,d,q) \le C_q d^{-1/2}$; thus we have $\alpha(2,X,q) \ge \frac{1}{2}$ for all $q$ as was expected.

\end{document}